\documentclass{gtmon_a}
\pdfoutput=1

\usepackage{pinlabel}

\proceedingstitle{Heegaard splittings of 3--manifolds (Haifa 2005)}
\conferencestart{10 July 2005}
\conferenceend{19 July 2005}
\conferencename{Heegaard splittings of 3--manifolds}
\conferencelocation{Haifa}

\editor{Cameron Gordon}
\givenname{Cameron}
\surname{Gordon}

\editor{Yoav Moriah}
\givenname{Yoav}
\surname{Moriah}

\title[Problems]{Heegaard splittings of 3--manifolds (Haifa 2005)\\Problems}

\author{Cameron McA Gordon (Editor)}  
\givenname{Cameron McA}
\surname{Gordon}
\address{University of Texas at Austin\\
Department of Mathematics\\\newline
1 University Station -- C1200\\
Austin, TX 78712-0257\\USA}

\email{gordon@math.utexas.edu}
\urladdr{}

\keyword{Problems}
\keyword{Heegaard splitting}
\keyword{3--manifold}

\subject{primary}{msc2000}{57M27}
\subject{primary}{msc2000}{57M25}
\subject{secondary}{msc2000}{57M07}
\volumenumber{12}
\issuenumber{}
\publicationyear{2007}
\papernumber{16}
\startpage{401}
\endpage{411}
\doi{}
\MR{}
\Zbl{}

\received{22 October 2007}
\revised{}
\accepted{}
\published{3 December 2007}
\publishedonline{3 December 2007}
\proposed{}
\seconded{}
\corresponding{}
\version{}

\makeatletter
\def\cnewtheorem#1[#2]#3{\newtheorem{#1}{#3}[section]
\expandafter\let\csname c@#1\endcsname\c@thm}

\newtheorem{thm}{Theorem}[section]
\cnewtheorem{conject}[thm]{Conjecture}
\newtheorem{quest}{Question}[section]
\newtheorem{prob}{Problem}[section]
\makeatother  
\newcommand{\intersect}{\cap}
\newcommand{\union}{\cup}
\newcommand{\boundary}{\partial}
\newcommand{\CC}{{\cal C}}
\begin{document}

\begin{htmlabstract}
These are problems on Heegaard splittings, that were raised at the
Workshop, listed according to their contributors: David Bachman, Mario
Eudave-Mu&#241;oz, John Hempel, Tao Li, Yair Minsky, Yoav Moriah and
Richard Weidmann.  On pages 285&ndash;298 of this monograph Hyam Rubinstein
gives a personal collection of problems on 3&ndash;manifolds.
\end{htmlabstract}

\begin{abstract}
These are problems on Heegaard splittings, that were raised at the
Workshop, listed according to their contributors: David Bachman, Mario
Eudave-Mu\~noz, John Hempel, Tao Li, Yair Minsky, Yoav Moriah and
Richard Weidmann.  In \cite{Rub} Hyam Rubinstein gives a personal
collection of problems on 3--manifolds.
\end{abstract}

\begin{asciiabstract}
These are problems on Heegaard splittings, that were raised at the
Workshop, listed according to their contributors: David Bachman, Mario
Eudave-Munoz, John Hempel, Tao Li, Yair Minsky, Yoav Moriah and
Richard Weidmann.  On pages 285-298 of this monograph Hyam Rubinstein
gives a personal collection of problems on 3-manifolds.
\end{asciiabstract}
\begin{webabstract}
These are problems on Heegaard splittings, that were raised at the
Workshop, listed according to their contributors: David Bachman, Mario
Eudave-Mu\~noz, John Hempel, Tao Li, Yair Minsky, Yoav Moriah and
Richard Weidmann.  On pages 285--298 of this monograph Hyam Rubinstein
gives a personal collection of problems on 3--manifolds.
\end{webabstract}

\maketitle

\section{David Bachman}

We analyze what is known, and what is not known, about the following question: 

\begin{quest}
\label{MainQuestion}
Which manifolds have infinitely many Heegaard splittings of the 
same genus, and how are the splittings constructed?
\end{quest}

The only known examples of 3--manifolds that admit infinitely many Heegaard 
splittings of the same genus are given in Sakuma \cite{sakuma},
Morimoto and Sakuma \cite{ms:91} and Bachman and Derby-Talbot \cite{bdt}.  
By Li's proof of the Waldhausen conjecture \cite{Li}, any such manifold must 
have an essential torus $T$. 
In the presence of such a torus, it is easy to see how an infinite number of 
Heegaard splittings of the same genus might arise. 
Simply take any splitting that intersects $T$, and Dehn twist the splitting 
about $T$. But is this the only way such an infinite collection might arise? 
What can be said about the manifold if this construction does {\it not} 
produce an infinite collection of non-isotopic splittings?

We analyze what is known when the torus is separating. 
(Some of the results described below hold in the non-separating case as well.)
Suppose then $T$ is an essential torus which separates a closed, orientable, 
irreducible 3--manifold $M$ into $X$ and $Y$. Then we may think of $M$ as 
being constructed from $X$ and $Y$ by gluing by some homeomorphism, 
$\phi\co \partial X \to \partial Y$. 

Fix separate triangulations of $X$ and $Y$ (these do not have to agree in 
any way on $T$). By a result of Jaco and Sedgwick \cite{js:03} the sets 
$\Delta _X$ and $\Delta _Y$ of slopes on $\partial X$ and $\partial Y$ that 
bound normal or almost normal surfaces in $X$ and $Y$ are finite 
(see Bachman \cite{bachman:01} for a further discussion of the almost normal case). 
We may thus define the {\it distance  $d(T)$ of $T$} to be the distance 
between the sets $\phi(\Delta _X)$ and $\Delta _Y$, as measured by the path 
metric in the Farey graph of $\partial Y$.  

\begin{thm}[Bachman, Schleimer, Sedgwick \cite{BSS}]
If $d(T)\ge 2$ then every Heegaard splitting of $M$ is an amalgamation of 
Heegaard splittings of $X$ and $Y$. 
\end{thm}

The relevance of this theorem is that Dehn twisting an amalgamation about the 
torus $T$ will not produce non-isotopic Heegaard splittings. Hence, 
when $d(T) \ge 2$ the manifold $M$ can only admit an infinite collection 
of non-isotopic Heegaard splittings if $X$ or $Y$ does. The converse, 
however, is more subtle. That is, if we know $X$ or $Y$ has infinitely 
many non-isotopic Heegaard splittings of the same genus, then does it follow 
that $M$ does as well? We conjecture the following:

\begin{conject}
If $d(T) \ge 3$ and $X$ or $Y$ has infinitely many Heegaard splittings of the 
same genus then $M$ does as well.
\end{conject}

When $d(T)=0$ or $1$, there is the possibility that there is a strongly 
irreducible Heegaard splitting of $M$. Let $H$ denote such a splitting 
surface. By a classic result of Kobayashi, $H$ can be isotoped to meet 
$T$ in a non-empty collection of loops that are essential on both surfaces. 

One way Dehn twisting $H$ about $T$ can fail to produce a non-isotopic 
Heegaard splitting is if $H \cap X$ is a fiber of a fibration of $X$. 
Then the effect of the Dehn twist can be ``undone" by pushing $H \cap X$ 
around the fibration. This is precisely what happens when $T$ is a 
separating vertical torus in a Seifert Fibered space. 
(For a complete resolution of \fullref{MainQuestion} in the case of 
Seifert Fibered spaces, see \cite{bdt}.) 
A second thing that may happen is that $H \cap X$ is the union of two pages 
of an open book decomposition of $X$. Then the effect of the Dehn twist 
can be undone by spinning $H \cup X$ about the open book decomposition. 
We conjecture that these are the only ways that Dehn twisting $H$ about 
$T$ can fail to produce non-isotopic Heegaard splittings:

\begin{conject}
If $M$ admits an essential torus $T$, separating $M$ into $X$ and $Y$, and 
a strongly irreducible Heegaard splitting $H$, then either Dehn twisting 
$H$ about $T$ produces an infinite collection of non-isotopic Heegaard 
splittings, or $H$ can be isotoped so that $H \cap X$ or $H \cap Y$ is 
either a fiber of a fibration of $X$ or $Y$ or two pages of an open book 
decomposition of $X$ or $Y$. 
\end{conject}

\section{Mario Eudave-Mu\~noz}

Let $F$ be a closed surface of genus $1$ standardly embedded in $S^3$,
that is, it bounds a solid torus on each of its sides.  We say that a
knot $K$ has a $(1,b)$--{\it presentation} or that it is in a
$(1,b)$--{\it position}, if $K$ has been isotoped to intersect $F$
transversely in $2b$ points that divide $K$ into $2b$ arcs, so that
the $b$ arcs in each side can be isotoped, keeping the endpoints
fixed, to disjoint arcs on $F$.  The {\it genus--1--bridge number} of
$K$, $b_1(K)$, is the smallest integer $n$ for which $K$ has a
$(1,n)$--presentation.  We say that a knot is a $(1,n)$--{\it knot} if
$b_1(K)\leq n$.  If $K$ is a $(1,1)$--knot, then it is easy to see
that $K$ has tunnel number one. On the other hand, if $K$ has tunnel
number one, it seems to be very difficult to determine $b_1(K)$.  It
has been of interest to find tunnel number one knots with large
genus--1--bridge number, see for example Kobayashi and Rieck
\cite{KR2}.

Moriah and Rubinstein \cite{MR} showed the existence of tunnel number one 
knots $K$ with $b_1(K)\geq 2$. Morimoto, Sakuma and Yokota also showed this, 
and gave explicit examples of 
knots $K$, with tunnel number one and $b_1(K)=2$ \cite{MSY}. 
It was shown by Eudave-Mu\~noz  \cite{E3},  that many of the tunnel  number 
one knots $K$ constructed in \cite{E1}  are not $(1,1)$--knots; this was 
extended in \cite{E4}, where it is shown that many  such knots are 
not $(1,2)$--knots. We remark that such knots can be explicitly described, 
for example the knot $K$ shown in Figure 13 of \cite{E4} satisfies 
$3\leq b_1(K)\leq 4$. Combining results of \cite{E2} and \cite{ER}, it is 
also possible to give explicit examples of knots $K$, with tunnel number 
one and $b_1(K)=2$, for example the knot shown in Figure 4 of \cite{E2}.
Valdez-S\'anchez and Ram\'{\i}rez-Losada have also shown explicit examples of 
tunnel number one knots $K$ with $b_1(K)=2$ (personal communication). 
These knots bound punctured Klein bottles but are not contained in the 
$(1,1)$--knots bounding Klein bottles determined by the same authors \cite{RV1}.

Johnson and Thompson \cite{JT}, and independently Minsky, Moriah and 
Schleimer \cite{MMS} have shown that for any given $n$, there exist tunnel 
number one knots which are not $(1,n)$--knots. The two papers use similar 
techniques, prove the existence of such knots, but do not give explicit 
examples. 
For a tunnel number one knot, let $\Sigma$ be the Heegaard splitting of the 
knot exterior determined by the unknotting tunnel, and let $d(\Sigma)$ 
denote the Hempel distance of the splitting \cite{H}.
In \cite{JT} and \cite{MMS} the existence of tunnel number one 
knots with large Hempel distance is shown, and then results of Scharlemann and 
Tomova \cite{ST}, \cite{T} are used to ensure the knots have large 
genus--1--bridge number. 

We propose the following problems:

\begin{prob} For a given integer $n$, give explicit examples of tunnel 
number one knots $K$ with $b_1(K)\geq n$.
\end{prob}

\begin{prob} For a given integer $N$, give explicit examples of tunnel 
number one knots with $d(\Sigma)\geq N$.
\end{prob}

\begin{prob} For a given integer $n$, give explicit examples of tunnel 
number one knots $K$ with $b_1(K)\geq n$, but with bounded Hempel distance, 
say with $d(\Sigma)= 2$.
\end{prob}
  
\section{John Hempel}

\subsection*{Questions on the curve complex}

Let  $S$ be a surface and $C(S)$ be its  curve complex  whose vertices are 
the isotopy classes of essential simple closed curves in $S$ and whose 
$n$--simplexes are determined by $n+1$ distinct vertices with pairwise 
disjoint representatives. The {\it distance}  between vertices is the 
number of edges in a minimal edge path joining them. 
Understanding this distance function is a daunting task -- it is hard to 
tell where one is headed as one moves away from a given vertex:

\begin{quest} 
Given vertices $x, y$ of $C(S)$ is there an ``easy'' way to find a vertex $x_1$ with $x \cap x_1 = \emptyset $ and $d(x_1, y) = d(x,y) -1$?
\end{quest}

One answer is given by K Shackleton  \cite{S},
but it involves extending the curves to multicurves which satisfy an 
additional ``tightness" condition and requires a search space whose size 
grows very rapidly with the complexity of the problem. 
Also, the complexity is measured in terms of the intersection numbers of the 
multicurves  involved. 
This may not be the most natural measure and could be a distraction. 
I am hoping for something easier. The difficulty will be in making a proof.

If $d(x,y) > 2$, then  $x \cup y$ splits $S$ into contractible regions, 
each with an even number of edges, alternating between $x$ and $y$. 
We assume there are no bigons.  
Euler characteristic calculations yield that most of these regions are 
squares. When we look in $T$, a component of  $S$ split along $x$, we see 
families of parallel arcs from $y$ successively in the boundaries of these 
squares whose unions we call $y$--{\it stacks}. $T$ deforms to a graph 
$\Gamma$ with one vertex in each large region of $S- y$ and one edge 
crossing each $y$--stack (in $T$). Choose a maximal tree $\Delta$ in $\Gamma$. 
This determines a free basis for $\pi_1(T)$ whose elements are represented 
by simple closed curves each of which crosses exactly one of the $y$--stacks 
corresponding to an edge of $\Gamma - \Delta$ and crossing it once. 
We call any such  basis a $y$--determined free basis and its elements  
$y$--determined free generators. They are in many ways the most simple, 
relative to $y$,  curves in $S-x$. It is natural to ask:

\begin{quest}\label{Q2}  
Is there some $y$--determined free generator $x_1$ for $\pi_1(T)$ 
with $d(x_1,y)$ $< d(x,y)$? 
Does every $y$--determined free basis contain such an element?
\end{quest}

Caution: there are examples in which  not  every $y$--determined free 
generator, $x_1$ satisfies $d(x_1,y) < d(x,y)$.  
If \fullref{Q2} can't be answered, then:

\begin{quest}  
In terms of word length, in a $y$--determined free basis, how far must we 
look in order to find a simple closed curve $x_1 \subset T$ with 
$d(x_1,y) < d(x,y)$?
\end{quest}


\section{Tao Li}

Let $M_1$ and $M_2$ be two manifolds with connected boundary 
$\partial M_1\cong\partial M_2\cong S$ and $\phi\co \partial M_1\to\partial M_2$ 
a homeomorphism.  If one glues $M_1$ to $M_2$ by identifying $x$ to $\phi(x)$ 
for each $x\in\partial M_1$, one obtains a closed manifold $M$.  
We say that $M$ is an amalgamation of $M_1$ and $M_2$.  

Let $\mathcal{D}_i$ ($i=1,2$) be the set of curves in $\partial M_i$ that 
bound essential disks in $M_i$.  We propose the following conjecture.  

\begin{conject}\label{Camal}
There is an essential curve $\mathcal{C}_i$ ($i=1,2$) in $\partial M_i$ such 
that if the distance between $\mathcal{D}_2\cup\mathcal{C}_2$ and 
$\phi(\mathcal{D}_1\cup\mathcal{C}_1)$ in the curve complex 
$\mathcal{C}(S)$ ($S=\partial M_2$) is sufficiently large, then either 
the minimal-genus Heegaard splitting of $M=M_1\cup_\phi M_2$ is an 
amalgamation or $S$ itself is a minimal-genus Heegaard surface in which 
case both $M_1$ and $M_2$ are handlebodies.
\end{conject} 

\fullref{Camal} is a generalization of two recent theorems.  
In the case that $M_1$ and $M_2$ are atoroidal and have incompressible 
boundaries, ie, $\mathcal{D}_1=\mathcal{D}_2=\emptyset$, the conjecture is 
proved by Souto \cite{So} and Li \cite{L11}.  Note that \cite{L11} also 
gives an algorithm to find $\mathcal{C}_1$ and $\mathcal{C}_2$.  

In the case that both $M_1$ and $M_2$ are handlebodies, $\mathcal{C}_1$ 
and $\mathcal{C}_2$ can be chosen to be empty and \fullref{Camal} 
follows from a recent theorem of Scharlemann and Tomova \cite{ST}.  
The theorem of Scharlemann and Tomova can be formulated as: if the distance 
between $\mathcal{D}_2$ and $\phi(\mathcal{D}_1)$ (ie, the Hempel distance) 
is large, then the genus of any other Heegaard splitting must be large unless 
it is a stabilized copy of $S$.

\section{Yair Minsky}

For a handlebody $H$ we have an inclusion of mapping class groups,
$MCG(H) < MCG(\boundary H)$. If $M = H_+ \union_S H_-$ is a Heegaard
splitting we denote $\Gamma_\pm = MCG(H_\pm)$ $< MCG(S)$, and moreover
let $\Gamma^0_\pm$ be the kernel of the map $MCG(H_\pm) \to$\break
${\rm Out}(\pi_1(H_\pm))$ (ie, $\Gamma^0_\pm$ is the group of mapping classes
of $H_\pm$ that are {\em homotopic} to the identity on $H_\pm$). 

\begin{quest}
When is $\Gamma_+\intersect \Gamma_-$ finite? 
\ldots\ finitely generated? 
\ldots\ finitely presented?
\end{quest}

\begin{quest}
When is $\langle\Gamma_+,\Gamma_-\rangle$ equal to the
amalgamation $\Gamma_+ *_{\Gamma_+\intersect \Gamma_-} \Gamma_-$? 
\end{quest}

\begin{quest}
When  is the map $\Gamma_+\intersect\Gamma_- \to MCG(M)$ injective? 
\end{quest}

When $M=S^3$ and ${\rm genus}(S)=2$, Akbas \cite{akbas} proved 
$\Gamma_+\intersect \Gamma_-$ is finitely presented.
(Finitely generated has a longer history starting with Goeritz \cite{Goe}
-- for details see Scharlemann \cite{scharlemann:automorphisms}.)
\medskip

Let $\Delta_\pm\subset \CC(S)$ be the set of (isotopy classes of)
simple curves in $S$ which bound disks in $H_\pm$. Let
$Z\subset\CC(S)$ be the simple curves in $S$ that are homotopic to the
identity in $M$. Note that $Z$ contains $\Delta_\pm$, and is invariant
under $\Gamma^0_\pm$. 

Namazi \cite{namazi:finite} showed if the distance of the splitting
(${\rm dist}_{\CC(S)}(\Delta_+,\Delta_-)$) is
sufficiently large then $\Gamma_+\intersect \Gamma_-$ is finite. Note
also the Geometrization Theorem  plus Hempel \cite{H}, 
plus Mostow rigidity plus Gabai--Meyerhoff--Thurston
\cite{gabai-meyerhoff-thurston}, implies that  
the image of $\Gamma_+\intersect \Gamma_-$ in $MCG(M)$ is finite when
the splitting distance is at least 3. 

\begin{quest}
When is $Z$ equal to the orbit 
$\langle\Gamma^0_+,\Gamma^0_-\rangle(\Delta_+\union\Delta_-)$?
\end{quest}

{\bf Remarks}\qua When $M=S^3$ or a connected sum of $S^2\times S^1$'s, with
the natural splitting, equality holds trivially. If $M$ is a lens
space $L_{p,q}$, with $p\ge 2$,  and $S$ is the torus, then they are not equal --
$Z$ is all of $\CC(S)=\Q\union\{\infty\}$, whereas the orbit of
$\Delta_+\union\Delta_- = \{\frac01,\frac{p}{q}\}$ under
$\langle\Gamma^0_+,\Gamma^0_-\rangle<SL(2,\Z)$ is strictly smaller than
$\Q\union\{\infty\}$. 
One might reasonably ask if there is equality when $M$
is hyperbolic, or if the splitting distance is
sufficiently large. 

\section{Yoav Moriah}
    
\begin{quest}
Give an example of a weakly reducible but non-stabilized non-minimal 
Heegaard splitting of a closed 3--manifold.  
\end{quest}
  
There are plenty of examples for manifolds with strongly irreducible 
Heegaard splittings of arbitrarily high genus. 
However there are no examples of manifolds (closed or with a single boundary 
component) with  a weakly
reducible but non-stabilized non-minimal Heegaard splitting.
  
\begin{quest}
Prove that if a manifold has strongly irreducible Heegaard splittings
of arbitrarily high genus then the Heegaard splittings are of the form
$H + nK$ where $H$ and $K$ are surfaces with $K$ perhaps not
connected. (See Moriah, Schleimer and Sedgwick \cite{MSS} and
Li~\cite{Li}.)  Or give a counterexample.
\end{quest}
  
\begin{quest}
There are examples by Kobayashi and Rieck (see below) of a 
$3$--manifold which has both a weakly reducible and strongly irreducible 
minimal genus Heegaard splitting.  These examples are
very   particular. Are there other such examples of a different nature?
\end{quest}
 
\begin{thm}[T Kobayashi and Y Rieck \cite{KR}]  
There are infinitely many $3$--manifolds which have both strongly irreducible 
and weakly  reducible Heegaard splittings of minimal genus.
\end{thm}
 
\begin{proof}    Let $X = S^3 - N(K_1)$,
$Y = S^3 - N(K_2)$ and  $Z = S^3 - N(K_3)$, 
where $K_1 = T_{(2,3)}$ the trefoil knot,
$K_2 = L(\alpha, \beta)$  with $ \alpha $ even, 
is any $2$--bridge link which is not the 
Hopf link and $K_3 = K(2,5) $ is the figure $8$ knot.  
Let $\mu_1$ and $\mu_2$ be meridians 
of  the $2$--bridge link $L(\alpha, \beta)$,  
$\lambda$ be the longitude of $K_3$ and 
$\gamma$ be the boundary of the annulus in $X$.
  
Attach $X$ and $Z$ to $Y$ by gluing their tori boundaries 
so that $\gamma$ is mapped to $\mu_1$  and $\lambda$ to $\mu_2$.  
We obtain a closed $3$--manifold $M$ with incompressible  tori. 
As the genus two toroidal $3$--manifolds are classified by Kobayashi \cite{Ko}
this manifold cannot  have genus~2.  
The surface $S$ obtained from the bridge sphere  $\Sigma$ union the annulus in 
the trefoil complement and the two genus one Seifert  surfaces of the 
figure $8$ complement is a  closed surface of genus 
$g = 1 - \frac {\chi (S) } 2 =  1 -  ((-2) + (-2) + 0)/2 = 3$. Hence it is a 
minimal genus Heegaard splitting if it is a Heegaard surface.
  
Let $V_1, V_2$ be the two components of $S^3 - N(K_1)$. 
So on one side of $S$ we have at the first stage the solid torus $V_1$, 
say, glued to  the genus two handlebody $W_1$,
which is one of the two components $W_1,W_2$ of $Y- \Sigma$, along a primitive 
meridional annulus to obtain a genus two handlebody.  
This handlebody is glued in the second 
stage to the genus two handlebody which is a regular neighborhood 
of the Seifert surface along a primitive  meridional annulus. 
So we get a genus three handlebody  $U_1$.  
  
On the  other side of $S$ we  have $V_2$ glued to  the genus 
two handlebody $W_2$ along a primitive meridional annulus to obtain 
a genus two handlebody which is then glued again along a primitive meridional 
annulus to a genus two handlebody which is a regular   
neighborhood of the Seifert surface of the figure $8$ knot complement. 
Thus we get a genus three handlebody $U_2$  and $(U_1,U_2)$ is a 
genus three Heegaard splitting for $M$. 
  
In \cite[Proposition 3.1]{Ko} Kobayashi proves that a Heegaard
splitting of the form $(U_1,U_2)$ is always strongly irreducible if
the link $L(\alpha,\beta)$ is not trivial or a Hopf link.

Now consider the union along the torus boundary of $X$ and $Y$. This 
is a manifold with a minimal genus two Heegaard splitting and with one 
torus boundary component. Hence when this Heegaard splitting is amalgamated 
with the genus two Heegaard splitting of $Z$ we obtain a genus three Heegaard 
splitting for $M$ which is weakly reducible as it is obtained by amalgamation.
\end{proof}

\begin{conject}
Let $K_1$ and $K_2$ be { \it prime} knots in $S^3$ then 
$t(K_1 \# K_2) \leq t(K_1) + t(K_2)$ if and only if one of $K_i$ 
has a minimal genus Heegaard splitting
with primitive meridian. (See Moriah~\cite{Mo}.)
\end{conject}

Recently T Kobayashi and Y Rieck disproved the conjecture for 
knots which are not prime (see~\cite{KR1}). 
The conjecture is known by work of Morimoto for tunnel number 
one knots and for knots which are connected sums of two prime knots 
each of which is also $m$--small. (See~\cite{Mo} for more references.)

\begin{conject}
Given a knot  $K \subset S^3$ which is not  $\gamma$--primitive
then a boundary stabilization of a minimal genus Heegaard splitting 
of $E(K)$ is non-stabilized.
\end{conject}

\begin{conject}
All twisted torus knots of type $K = T(p,q,2,r)$ which are not 
$\mu$--primitive have a unique  
(minimal) genus  two Heegaard splitting. (See Moriah and Sedgwick~\cite{MS}.)
\end{conject}

\begin{conject}
What are the properties of meridional essential surfaces which ensure that the 
tunnel number degenerate? Can these surfaces be classified?
\end{conject}

\begin{quest} 		
Are there knots which are not  $K_1 = K^n(-2,3,-3,2)$  and $2$--bridge 
knots so that $t(K_1 \# K_2)  <   t(K_1) + t(K_2)$?  (See Morimoto~\cite{Mrr}.)
\end{quest}

\begin{quest} 	
Does rank equal genus for  hyperbolic $3$--manifolds?
\end{quest}

In~\cite{LM} Lustig and Moriah defined a condition on complete disks systems 
in a Heegaard splitting, called the {\it double rectangle condition}. 
It was  shown that if a manifold has a Heegaard splitting which has some 
complete disk system which satisfies this condition then there are only 
finitely many such disk systems with that property. However the 
double rectangle condition is clearly non-``generic". 
Is it possible to define some other condition which will be ``generic'' 
in some reasonable sense? The intuition is that if the Heegaard distance 
of the splitting is sufficiently high then some form of this might be possible.

\section{Richard Weidmann}

Let $(M;V,W)$ be a Heegaard splitting of genus $n$. 
Let $g_1,\ldots,g_n\in \pi_1(M)$ be the elements corresponding to a 
spine of $V$.

\begin{quest}
When is $\langle g_1,\ldots,g_{n-1}\rangle$ a free group of rank $n-1$?
\end{quest}

(This would be a kind of ``Freiheitsatz''.) 
For instance, what happens with the Casson--Gordon examples of strongly 
irreducible splittings of a fixed $M$ of arbitrarily high genus?

Note that the freeness holds for the examples exhibited by Namazi \cite{N}
and Namazi--Souto \cite{NS}.

\begin{quest}
Find (3--manifold) groups that have only finitely many irreducible Nielsen
equivalence classes of generating tuples.
\end{quest}

Here a $n$--tuple is called irreducible if it is not Nielsen equivalent to 
a tuple of type $(x_1,\ldots,x_{n-1},1)$.
Note that free groups and free Abelian groups have this property, in fact for 
those groups there is precisely one irreducible Nielsen equivalence class.
Note that for Heegaard splittings it has been shown by Tao Li that closed 
non-Haken 3--manifolds only have finitely many isotopy classes of irreducible 
Heegaard splittings so those groups might be potential examples.

\bibliographystyle{gtart}
\bibliography{link}

\end{document}